\documentclass[12pt]{article}
\usepackage[dvipsnames,usenames]{color}
\usepackage{amsmath}
\usepackage{amsfonts}
\usepackage{amssymb}
\usepackage{mathrsfs}
\usepackage{amsthm}

\newcommand{\CC}{{\mathbb C}}
\newcommand{\RR}{{\mathbb R}}

\newcommand{\HH}{{\mathbb H}}
\newcommand{\BB}{{\mathbb B}}

\newcommand{\ov}{\overline}
\newcommand{\p}{\partial}
\newcommand{\w}{\widetilde}
\newcommand{\al}{\alpha}
\newcommand{\be}{\beta}
\newtheorem{thm}{Theorem}[section]

\newtheorem{lem}[thm]{Lemma}

\numberwithin{equation}{section}

\begin{document}
\title{\bf The CR immersion into a sphere with the degenerate CR Gauss map }
\author{Wanke Yin\footnote{partially supported by NSFC-11571260 and NSFC-11722110 }, Yuan Yuan\footnote{partially supported by Simons Foundation grant (\#429722 Yuan Yuan) and CUSE Grant Program at Syracuse University} and Yuan Zhang\footnote{partially supported by NSF DMS-1501024}}
\date{}
\maketitle

\begin{abstract}
It is a classical problem in algebraic geometry  to characterize the algebraic subvariety by using the Gauss map. In this note, we try to develop the analogue theory in CR geometry. In particular, under some assumptions, we show that a CR map between spheres is totally geodesic if and only if the CR Gauss map of the image is degenerate.

\end{abstract}

\section{Introduction}

Denote by $\mathbb{C}\mathbb{P}^n$ the complex projective space, and denote by $G(k, n)$ the Grassmannian of $\mathbb{C}\mathbb{P}^k$'s in $\mathbb{C}\mathbb{P}^n$. Let $V$ be  a complex analytic subvariety in $\mathbb{C}\mathbb{P}^n$ and $V_{sm}$ denotes its smooth points. Then the Gauss map of $V\subset \mathbb{C}\mathbb{P}^n$ is defined by $\gamma: V_{sm} \rightarrow G(k, n)$, which sends each smooth point $x\in V_{sm}$ to the projective tangent space $T_x(V)$. $\gamma$ is said to be {\it degenerate} if its generic fibers have positive dimensional components. Otherwise, $\gamma$ is called non-degenerate.
In Cartan's moving frame theory, the Gauss map has wide geometric applications in Euclidean and projective geometry. For example, one can obtain rigidity results from the degeneracy of the Gauss maps. In fact, the study of subvarieties of complex projective spaces, tori and hyperbolic space forms with degenerate Gauss maps are classical works due to Griffiths-Harris \cite{GH}, Ran \cite{R} and Hwang \cite{Hw}. The interested reader are referred to \cite{IL} for more recent progress on subvarieties of complex projective spaces with degenerate Gauss maps.

The Gauss map also has the close relation to the second fundamental form as the latter may be interpreted as the derivative of the Gauss map.
In CR geometry, the CR second fundamental form appeared in the fundamental work of
Chern-Moser \cite{CM} and Webster \cite{W1}, as well as the work of
Ebenfelt-Huang-Zaitsev \cite{EHZ} in the study of the classification and rigidity of CR submanifolds.
One of the central problems in CR geometry is the classification of smooth CR maps between spheres. This problem has been extensively studied and many important progresses have been made by many authors in recent years (cf. \cite{W2,Fa86,Hu99,Ha05,HJX06,DL,HJY14,Eb} and references therein).
If the CR second fundamental form vanishes, Ji and the second author showed that
the smooth immersed
strongly pseudoconvex real hypersurface in a sphere $\partial\mathbb{B}^n$  must be linear \cite{JY}. Cheng-Ji
later relaxed the condition to the difference of the second fundamental form and CR second fundamental form vanishing and proved the linearity under some codimension restriction \cite{CJ}.
However, in the CR geometry, 
the Gauss map is not fully understood.
One can define the Gauss map for any $C^1$ immersed  CR submanifold in $\p \mathbb{B}^N$ as the sphere $\p \mathbb{B}^N$ may be embedded into $\mathbb{C}\mathbb{P}^N$ (The detailed formulation of the CR Gauss map is given in the last paragraph of the next section). The following interesting question is formulated in  \cite{CJL}: Let $V\subset \p \mathbb{B}^N$ be an immersed spherical
CR submanifold.  Is the CR Gauss map $\gamma$ degenerate
if and only if $V$ is the image of a linear embedding $F : \p \mathbb{B}^n \rightarrow \p \mathbb{B}^N$?
In \cite{CJL}, Cheng-Ji-Liu answered the question in the following two cases: (1) dim$_{\mathbb{R}}V=3,\ N=3$; (2) $V = F(\p \mathbb{B}^2)$ and
$F : \p \mathbb{B}^2 \rightarrow \p \mathbb{B}^N$ is the restriction of a rational holomorphic map with deg$(F) = 2$.


We next state our main results, in which the terminology
will be defined in the next section.

\begin{thm}\label{main}
Let $F:\p \mathbb{B}^n \rightarrow \p \mathbb{B}^N$ be a $C^3$-smooth CR
  map with geometric rank $\kappa_0\le n-2$. Assume that one of the following conditions hold:

 (1)   the  degeneracy rank $\leq$ 2,

 (2) the third degeneracy rank $\geq$ 3, and the third degeneracy dimension $$d_{3}\neq \frac{\kappa_0}{6}\Big(3(\kappa_0+3)n-(\kappa_0+1)(2\kappa_0+1)\Big).$$
  Then the CR Gauss map of $F(\p \mathbb{B}^n)$ in $\p \mathbb{B}^N$ is degenerate if and only if $F$ is a totally
 geodesic embedding.
\end{thm}

As an immediate consequence, we obtain

\begin{thm}\label{main1}
Let $F:\p \mathbb{B}^n \rightarrow \p \mathbb{B}^N$ be a $C^3$-smooth CR
  map with geometric rank $\kappa_0\le n-2$.
  Suppose that $N<\frac{1}{2}(\kappa_0+1)(\kappa_0+2)n-\frac{1}{6}\kappa_0(\kappa_0+1)(2\kappa_0+1)$.
 Then the CR Gauss map of $F(\p \mathbb{B}^n)$ in $\p \mathbb{B}^N$ is degenerate if and only if $F$ is a totally
 geodesic embedding.
\end{thm}

\section{Notations and Preliminaries}
In this section, we start by recalling some notations and properties
associated to the proper holomorphic maps between balls, which are
established in [Hu99,03] and \cite{HJX06}. Next, we define the CR Gauss
maps of these maps and reduce the condition on the CR Gauss maps to a
proper form, following the lines of \cite{Hw} and \cite{CJL}.

 Let $\p\mathbb{B}^n$ be the sphere in $\mathbb{C}^n$ and write $\p\HH_n:=\{(z,w)\in {\CC}^{n-1}\times {\CC}:
  \ \hbox{Im}(w)=|z|^2\}$ for the Heisenberg group.
By the Cayley transformation
\begin{equation}
\rho_n: {\HH}_n\to {\BB}^n, \ \ \rho_n(z,w)=\bigg(\frac{2z}{1-iw},\
\frac{1+iw}{1-iw}\bigg) \label{eqn:rho}
\end{equation}
  we can identify a CR map $F$ from $\p{\BB}^n$ into $\p{\BB}^N$
  %
  with
$\p\rho^{-1}_N\circ F\circ \p\rho_n$, which is a CR map from
$\p{\HH}_n$ into $\p{\HH}_N$.

Parameterize $\partial \HH_n$ by $(z,\overline{z},u)$ through the
map $(z,\overline{z},u)\to (z,u+i|z|^2)$. For a non-negative integer $m$ and a function $h(z,\overline{z},u)$
defined over a small ball  $U$ of $0$ in $\partial \HH_n$, we say   $h(z,\overline{z},u)=o_{wt}(m)$ if
$\frac{h(tz,t\overline{z},t^2u)}{|t|^{m}}\to 0$ uniformly for
$(z,u)$ on any
  compact subset of $U$ as $t(\in {\RR})\to 0$.
For a holomorphic function (or
  map) $H(z,w)$, we write
  $$H(z,w)=\sum_{k,l=0}^{\infty}H^{(k,l)}(z)w^l=\sum_{i_1\cdots,i_{n-1},l=0}^{\infty}H^{(i_1I_1+\cdots+i_{n-1}I_{n-1}+lI_n)}
  z_1^{i_1}\cdots z_{n-1}^{i_{n-1}}w^l.$$
Here  $H^{(k,l)}(z)$ a polynomial of degree $k$ in $z$.
\bigskip

Let $F=(f,\phi,g)=(\widetilde{f}, g)= (f_1,\cdots,f_{n-1},
\phi_1,\cdots, \phi_{N-n},g)$ be a non-constant $C^2$-smooth CR map
from $\partial{\HH}_n$ into $\partial{\HH}_N$ with $F(0)=0$. For
each $p=(z_0, w_0)\in M$ close to $0$, we write $\sigma^0_p\in
\hbox{Aut}(\HH_n)$ for the map sending $(z,w)$ to $(z+z_0, w+w_0+2i
\langle z,\overline{z_0} \rangle )$ and
$\tau^F_p\in\hbox{Aut}(\HH_N)$ by defining
$$\tau^F_p(z^*,w^*)=(z^*-\widetilde{f}(z_0,w_0),w^*-\overline{g(z_0,w_0)}-
2i \langle z^*,\overline{\widetilde{f}(z_0,w_0)} \rangle ).$$
  Then $F$ is
equivalent to
\begin {equation}
F_p=\tau^F_p\circ F\circ \sigma^0_p=(f_p,\phi_p,g_p).
\label{eqn:nor01}
\end{equation}

  Notice that
$F_0=F$ and $F_p(0)=0$. Let
\begin{equation}\begin{split}
&E_l(p)=(\frac{\p \w{f}_p}{\p z_l})\big|_0=\big(\frac{\p
{f}_{p,1}}{\p z_l},\cdots,\frac{\p {f}_{p,n-1}}{\p z_{l}},\frac{\p
{\phi}_{p,1}}{\p z_l},\cdots,\frac{\p {\phi}_{p,N-n}}{\p
z_l}\big)\big|_0=L_l(\w{f})(p),\\
&E_w(p)=(\frac{\p \w{f}_p}{\p w})\big|_0=\big(\frac{\p {f}_{p,1}}{\p
w},\cdots,\frac{\p {f}_{p,n-1}}{\p w},\frac{\p {\phi}_{p,1}}{\p
w},\cdots,\frac{\p {\phi}_{p,N-n}}{\p w}\big)\big|_0=T(\w{f})(p).
\end{split}\end{equation}
Then the rank of $\{E_1(p),\cdots,E_{n-1}(p)\}$ is $n-1$. Write
$\lambda(p)=g'_w(p)-2i\langle
\w{f}'_w(p),\ov{\w{f}(p)}\rangle=|L_j(\w{f})|^2$. Then we can choose
vectors $C_l(p)$ for $1\leq l\leq N-n$ such that
\begin{equation}\begin{split}\label{ap}
&A(p)=\Big(\frac{E^t_1(p)}{\sqrt{\lambda(p)}},\cdots,\frac{E^t_1(p)}{\sqrt{\lambda(p)}},
C_1^t(p),\cdots,C_{N-n}^t(p)\Big)
\end{split}\end{equation}
is a unitary matrix. Define
\begin{equation}\begin{split}
&F_p^*=(\w{f}^*_p,g_p^*)=\frac{1}{\sqrt{\lambda(p)}}F_p
\cdot\left(\begin{array}{cc} \ov{A^t(p)} & 0\\ 0 &
\frac{1}{\sqrt{\lambda(p)}}\end{array}\right).
\end{split}\end{equation}
Then $F_p^*$ has the following form:
\begin{equation}\begin{split}\label{ab}
f_j^*&=z_j+a_jw+O(|(z,w)|^2),\\
 \phi_j^*&=b_jw+O(|(z,w)|^2),\\
g^*&=w+dw^2+O(|zw|)+o(|(z,w)|^2).
\end{split}\end{equation}
Write $a=(a_1,\cdots,a_{n-1},b_1,b_{N-n})$, $b=(b_1,\cdots,b_{N-n})$
and define $F_p^{**}$ by
\begin{equation}\begin{split}
&\w{f}^{**}_p=\frac{1}{q^*(z,w)}\big(\w{f^*_p}(z,w)-ag^*_p(z,w)\big),\
\w{g}^{**}_p=\frac{1}{q^*(z,w)}g^*_p.
\end{split}\end{equation}
Here we have set
\begin{equation}\begin{split}
q^*(z,w)&=1+2i\ov{a}\w{f^*_p}+(r-i|a|^2)g^*_p(z,w),\
r=\frac{1}{2}\text{Re}\big(\frac{\p ^2 g_p^*}{\p w^2}(0)\big).
\end{split}\end{equation}

Then $F_p^{**}$ has the following normalization, which is of fundamentally important for the understanding of
the geometric properties of $F$.

\medskip
\begin{lem}[$\S 2$, Lemma 5.3, \cite{Hu99}]
    Let $F$ be a $C^2$-smooth CR map
from  $\partial\HH_n$ into $\partial\HH_N$, $2\le n\le N$. For each
$p\in \partial \HH_n$, there is an automorphism $\tau^{**}_p\in
  {Aut}_0({\HH}_N)$ such that
$F_{p}^{**}:=\tau^{**}_p\circ F_p$ satisfies the following
normalization:
$$f^{**}_{p}=z+{\frac{i}{ 2}}a^{**(1)}_{p}(z)w+o_{wt}(3),\ \phi_p^{**}
={\phi_p^{**}}^{(2)}(z)+o_{wt}(2), \ g^{**}_{p}=w+o_{wt}(4),\
\hbox{with}$$
$$\langle \overline{z}, a_{p}^{**(1)}(z)\rangle
|z|^2=|{\phi_p^{**}}^{(2)}(z)|^2.$$
\end{lem}

 Write
$\mathcal{A}(p)=-2i(\frac{\partial^2(f^{\ast\ast}_p)
  _l}{\partial z_j\partial w}|_0)_{1\leq j,l\leq (n-1)}$ in the above
lemma. In \cite{Hu03}, Huang defined the {\it geometric rank} of $F$ at $p$,  denoted by
  $Rk_F(p)$, to be the rank
  of the $(n-1)\times (n-1)$ matrix  $\mathcal{A}(p)$.
Now we can define the {\it geometric rank} of $F$ to be $\kappa_0(F)=max_{p\in
  \partial\HH_n} Rk_F(p)$.
 For a $C^2$ smooth CR map $F$ from $\p{\BB}^n$ into $\p{\BB}^N$,
 the {\it geometric rank} of the map $F$ is defined by the map
$\rho_N^{-1}\circ F \circ \rho_n.$
By \cite{Hu03}, $\kappa_0(F)$ depends only on the equivalence class
of $F$ and  $\kappa_0(F)\le n-2$ when $N<\frac{n(n+1)}{2}$.

Let $F_p^{***}$ be defined as follows:
\begin{equation}\begin{split}\label{fp***}
F_P^{***}=\Big({f}_p^{**}(zU,w)U^{-1},\phi_p^{**}(zU,w)U^*,g_p^{**}(zU,w)\Big).
\end{split}\end{equation}
When $\kappa_0\leq n-2$,  $F_p^{***}$ satisfies the following
normalizations:
\begin{equation}\begin{cases}\label{good}
f_{j}=z_j+\frac{i}{2}\mu_jz_jw+o_{wt}(3)\ \text{for}\
j\leq \kappa_0,\\
f_{j}=z_j+o_{wt}(3)\ \text{for}\
\kappa_0<j\leq n-1,\\
\phi_{jk}=\mu_{jk}z_jz_k+\sum\limits_{h=1}^{n-1}e_{h,jk}z_hw+d_{jk}w^2+O(|(z,w)|^3)\
\text{for}\
(j,k)\in \mathcal{S}_0,\\
\phi_{jk}=\sum\limits_{h=1}^{n-1}e_{h,jk}z_hw+d_{jk}w^2+O(|(z,w)|^3)\
\text{for}\
(j,k)\in \mathcal{S}_1,\\
\ g=w+o_{wt}(4).
\end{cases}\end{equation}
Here, for $1\le \kappa_0\le n-2$, we write ${\cal S} ={\cal S}_0\cup
{\cal S}_1$, the index set  for all components of $\phi$, where
${\cal S}_{0}=\{(j,l): 1\le j\leq \kappa_0, 1\leq l\leq n-1, j\leq
l\}$ and ${\cal S}_1=\{(j, l): j=\kappa_0+1, \kappa_0+1\le l \le
N-n-\frac{(2n-\kappa_0-1)\kappa_0}{2} \}$. Also,
$\mu_{jl}=\sqrt{\mu_j+\mu_l}\ for\ j<l\le \kappa_0$; and $\
\mu_{jl}=\sqrt{\mu_j}$ if $j\le \kappa_0<l$ or if $j=l\le \kappa_0$.

Let $\tau\in \text{Aut}_0(\mathbb{H}_n)$ and $\sigma\in
\text{Aut}_0(\mathbb{H}_n)$ be given by
\begin{equation}\begin{split}\label{st}
\sigma(z,w)=\frac{(z-cw,w)}{q(z,w)},\
\tau(z^*,w^*)=\frac{(z^*+(c,0)w^*,w^*)}{q^*(z^*,w^*)}
\end{split}\end{equation}
with
\begin{equation}\begin{split}
&q(z,w)=1+2i\langle \ov{c},z \rangle-i|c|^2w,\\
&q^*(z^*,w^*)=1-2i\langle \ov{c},z^* \rangle-i|c|^2w^*,\\
&c=(c_1,\cdots,c_{n-1}).
\end{split}\end{equation}
Then by suitably choosing $c_j$ for $1\leq j\leq \kappa_0$, we can
make $F_p^{****}=\tau \circ F_P^{***}\circ \sigma$ still has the
form (\ref{good}). Furthermore, we can make $\frac{\p^2 f_j}{\p
w^2}(0)=0$ for $1\leq j\leq \kappa_0$. In \cite{HJX06}, the authors
proved the following normalization theorem for maps with geometric
rank bounded by $n-2$, though only part of it is needed later:

\begin{thm}\label{thm3} 
Suppose that $F$ is a rational proper holomorphic map from ${\HH}_n$
into ${\HH}_N$, which has
 geometric rank  $1\le\kappa_0\le n-2$ with $F(0)=0$. Then there are
$\sigma\in \hbox{Aut}({\HH}_n)$ and
  $\tau\in \hbox{Aut}({\HH}_N)$ such that
  $\tau\circ F\circ \sigma $ takes
the following form, which is still denoted by $F=(f,\phi,g)$ for
convenience of notation:

\begin{equation}
\left\{
  \begin{array}{l}
  f_l=\sum_{j=1}^{\kappa_0}z_jf_{lj}^*(z,w),\ \ l\le\kappa_0,\\
  f_j=z_j,\ \  \kappa_0+1\leq j\leq n-1,\\
  \phi_{lk}=\mu_{lk}z_lz_k+\sum_{j=1}^{\kappa_0}z_j\phi^*_{lkj},\ \ (l,k)\in {\cal S}_0,\\
  \phi_{lk}=\sum_{j=1}^{\kappa_0}z_j\phi_{lkj}^*=O_{wt}(3),\ \
 (l, k)\in {\cal S}_1,\\
  g=w,\\
f_{lj}^*(z,w)=\delta_l^j+\frac{i\delta_{l}^j\mu_l}{2}w+b_{lj}^{(1)}(z)w+O_{wt}(4),
\ \ 1\le l\le \kappa_0,\ \mu_l>0,\\
\phi^*_{lkj}(z,w)=O_{wt}(2),\ \  (l,k)\in {\cal S}_1.\ \
  \end{array}\right.
\label{eqn:hao}
\end{equation}
Here, for $1\le \kappa_0\le n-2$, we write ${\cal S} ={\cal S}_0\cup
{\cal S}_1$, the index set  for all components of $\phi$, where
${\cal S}_{0}=\{(j,l): 1\le j\leq \kappa_0, 1\leq l\leq n-1, j\leq
l\}$ and ${\cal S}_1=\{(j, l): j=\kappa_0+1, \kappa_0+1\le l \le
N-n-\frac{(2n-\kappa_0-1)\kappa_0}{2} \}$. Also,
$\mu_{jl}=\sqrt{\mu_j+\mu_l}\ for\ j<l\le \kappa_0$; and $\
\mu_{jl}=\sqrt{\mu_j}$ if $j\le \kappa_0<l$ or if $j=l\le \kappa_0$.
\end{thm}

For later use, we will also set
$$
\phi^{(1,1)}(z)=\sum^{\kappa_0}_{j=1} e_j z_j\ \text{with}\ e_j \in
\CC^{\sharp({\cal S})},\ \phi^{(1,1)}_{kl}(z)=\sum^{\kappa_0}_{j=1}
e_{j,kl} z_j\ \text{with}\ (k,l)\in \mathcal{S} .
$$
\bigskip

Next we define the degeneracy rank for any smooth CR map $F$ from $\p \HH_n$ to $\p \HH_N$, which is an invariant integer introduced by Lamel \cite{L01a}(see also
\cite{EHZ} and \cite{Eb13}. In fact, the degeneracy rank is defined for more general maps).

 For any point $p\in
\p\HH_n$, we define an increasing sequence of linear subspaces
$E_{k}(p) \subset \CC^{N}$ for $F$,
\begin{equation}
E_{k}(p)=span_{\CC}\{L^\alpha \hat\rho_{\ov Z} \circ F(p)\ |\ \ |\alpha|\le k\}
\end{equation}
where $L^\alpha=L^{\alpha_1} ... L^{\alpha_{n-1}}$, $\alpha=(\alpha_1, ..., \alpha_{n-1})$,
$|\alpha|=|\alpha_1|+...+|\alpha_{n-1}|$, $L_j=\frac{\p }{\p z_j} + 2i \ov{z_j} \frac{\p}{\p w}$,
 $\hat \rho(Z, \ov Z)$ is the defining function of the real hypersurface $\p\HH_N$,
and $\hat \rho_{\ov Z}:=\ov\p \hat \rho$ is complex gradient of $\rho$. Notice that $\{L_j\}_{1\le j\le n-1}$ form a basis of tangent vector fields of $(1,0)$ along
$\p\HH_n$.

We define $d_1(p):=0$ and
\begin{equation}
\label{dk(p)}
d_k(p):= \dim_\CC \ E_k(p)/E_1(p).
\end{equation}
Then we have a sequence of dimensions $d_1(p)=0\le d_2(p)\le d_3(p) \le ... \le d_k(p)\le ...$.
Notice that the dimensions $d_j(p)$ is upper semi-continuous. By moving $p$ to a nearby point $p_0$ if necessary, we may assume that all $d_l(p)$ are locally
constant near $p_0$ and
\begin{equation}
\label{rank}
d_2(p)<d_3(p)<...<d_{l_0}(p)=d_{l_0+1}(p)=....
\end{equation}
for some $l_0$ with $1\le l_0\le N-n+1$.  In other words, there
exists an open subset $U$ of $\p\HH_n$ on which  all $d_l(p)$ are
locally constant near $p_0$ and (\ref{rank}) holds. By \cite{EHZ}, we
may call such $l_0$ the {\it degeneracy rank} of $F$, and call such
$d_{l_0}$ the {\it degeneracy dimension} of $F$. These definitions
depend on the open subset $U$. By minimizing $l_0$ among all such
open sets, we can define degeneracy rank $l_0$ of $F$ as an
invariant. we also call  $d_j(p)$ for $j\in [2,l_0]$ the $j$-th degeneracy dimension. The dimensions $d_j(p)$ $j=1,\cdots,l_0$ can be interpreted as ranks of the CR second fundamental form
of $f$ and its covariant derivatives. The interested reader is referred to \cite[Section 2]{Eb13} for more details.

\bigskip

We end this section by recalling the CR Gauss map formulated in \cite{CJL}.
Let $F:\p \mathbb{H}_n \rightarrow \p \mathbb{H}_N$ be a rational CR
map. Write $F(z,w)=(f(z,w),\phi(z,w),g(z,w))$  and set $L_j=\frac{\p}{\p z_j}+2i\ov{z_j}\frac{\p}{\p w}$ for $1\leq j\leq n-1$.
Since $F$ is a CR map, we have $\ov{L_j}f=\ov{L_j}\phi=\ov{L_j}g=0.$ Thus the matrix
\begin{equation}
\left(\begin{array}{cccc} &L_1f&\ L_1g &\ L_1\phi\\
&\vdots &\vdots&\vdots\\
&L_{n-1}f&\ L_{n-1}g&\ L_{n-1}\phi\\
&Tf&\ Tg&\ T\phi
\end{array}\right)
\end{equation}
represents an element in the Grassmanian $G(n,N)$, which is Gauss map associate to the map. By an action of a
non-singular $n\times n$ matrix, the element is equivalent to the
unique matrix representation $(I_{n\times n}\ G)$, where $I_{n\times
n}$ is the unit matrix and
\begin{equation}
G(z,w)=\left(\begin{array}{lll} &L_1f&\ L_1g\\
&\vdots &\vdots\\
&L_{n-1}f&\ L_{n-1}g\\
&Tf&\ Tg\\
\end{array}\right)^{-1}\cdot \left(\begin{array}{cccc} &L_1\phi_1&\ \cdots &\ L_1\phi_{N-n}\\
&\vdots &\vdots&\vdots\\
&L_{n-1}\phi_{1}&\ \cdots&\ L_{n-1}\phi_{N-n}\\
&T\phi_{1}&\ \cdots&\ T\phi_{N-n}
\end{array}\right)(z,w).
\end{equation}
The CR Gauss map of the image $F(\p \mathbb{H}_n)$ is defined by
$\gamma:p\rightarrow G(p)$ for any $p\in \p\mathbb{H}_n$.

\section{Reduction of the degeneracy of the CR Gauss map}

In this section, we will reduce the CR Gauss map condition to a proper
form, which is crucial to the proof of our main theorem.

\begin{thm} Let $F:\p \mathbb{H}_n\rightarrow \p \mathbb{H}_N$ be a
$C^2$ CR immersion, the geometric rank of $F$ at $0$ is
$\kappa_0\in [1,n-1]$ which is maximal. We further suppose that $F$
has the normal form (\ref{eqn:hao}). Then for any fixed
$p=(z_0,w_0)\in \p \mathbb{H}_n$ near the origin,  the CR Gauss map
equation $\gamma(z,w)=\gamma(z_0,w_0)$, for $(z,w)$ close to
$(z_0,w_0)$, expressed in terms of $F_p^{****}$, takes the following
form
\begin{equation}\begin{split}\label{pqp****}
&\frac{\p \phi^{****}_p}{\p z_j} =O(|(z,w)|^2),\ \  \frac{\p
\phi^{****}_p}{\p w} =O(|(z,w)|^2).
\end{split}\end{equation}
\end{thm}

For any fixed $p=(z_0,w_0)\in \p \mathbb{H}_n$ near the origin, set
$\w{z}=z+z_0$ and $\w{w}=w+w_0+2i\ov{z_0}z$. We also write
\begin{equation}
P(z,w)=\left(\begin{array}{lll} &L_1f&\ L_1g\\
&\vdots &\vdots\\
&L_{n-1}f&\ L_{n-1}g\\
&Tf&\ Tg\\
\end{array}\right)(\w{z},\w{w}),\
Q(z,w)=\left(\begin{array}{ll} &L_1\phi\\
&\vdots \\
&L_{n-1}\phi\\
& T\phi\\
\end{array}\right)(\w{z},\w{w}).
\end{equation}

Then CR Gauss map equation $\gamma(z,w)=\gamma(z_0,w_0)$
 is equivalent to
\begin{equation}\label{pq}
Q(z,w)=P(z,w)P^{-1}(0)Q(0)
\end{equation}

Next we would like to express the system (\ref{pq}) in terms of
$F_p^{****}$ introduced in the preceding section through the
following $5$ steps.
\bigskip

{\bf Step I. Express the CR Gauss map equation in terms of $F_p$}

 By the
construction of Huang,  $F_p$ defined by (\ref{eqn:nor01}) takes the
following form:
\begin{equation}\begin{split}\label{fpgp}
&\w{f}_p(z,w)=\w{f}(\w{z},\w{w})-\w{f}(z_0,w_0),\\
&g_p(z,w)=g(\w{z},\w{w})-\ov{g(z_0,w_0)}-2i\ov{\w{f}(z_0,w_0)}\w{f}(\w{z},\w{w}).
\end{split}\end{equation}
A direct computation shows that
\begin{equation}\begin{split}
&\frac{\p \w{f}_p}{\p z_j}(z,w)=\big(\frac{\p \w{f}}{\p
z_j}+2i\ov{z_{0j}}\frac{\p \w{f}}{\p w}\big)(\w{z},\w{w}), \\
&T\w{f}_p(z,w)=(T\w{f})(\w{z},\w{w}),\\
&(Tg_p)(z,w)=(Tg)(\w{z},\w{w})-2i\ov{\w{f}(z_0,w_0)}\big(T\w{f}\big)(\w{z},\w{w}).
\end{split}\end{equation}
Hence we infer
\begin{equation}\begin{split}\label{ljftf}
&\big(L_j \w{f}\big)(\w{z},\w{w})=L_j \w{f}_p(z,w), \
(T\w{f})(\w{z},\w{w})=T\w{f}_p(z,w),\\
&(Tg)(\w{z},\w{w})=Tg_p(z,w)+2i\ov{\w{f}(z_0,w_0)}T\w{f}_p(z,w).
\end{split}\end{equation}
Applying $T$ to these equations, we can further get
\begin{equation}\begin{split}\label{ljtf}
&\big(TL_j \w{f}\big)(\w{z},\w{w})=TL_j \w{f}_p(z,w), \
(T^2\w{f})(\w{z},\w{w})=T^2\w{f}_p(z,w),\\
&(T^2g)(\w{z},\w{w})=T^2g_p(z,w)+2i\ov{\w{f}(z_0,w_0)}T^2\w{f}_p(z,w).
\end{split}\end{equation}
By (\ref{fpgp}) and (\ref{ljftf}), we obtain
\begin{equation}\begin{split}
\frac{\p g_p}{\p z_j}(z,w)=&\Big(\frac{\p g}{\p
z_j}+2i\ov{z_{0j}}\frac{\p g}{\p
w}-2i\ov{\w{f}(z_0,w_0)}\big(\frac{\p \w{f}}{\p
z_j}+2i\ov{z_{0j}}\frac{\p \w{f}}{\p w}\big)\Big)(\w{z},\w{w})\\
=&(L_j g)(\w{z},\w{w})-2i\ov{z_j}(Tg)(\w{z},\w{w})
-2i\ov{\w{f}(z_0,w_0)}\big(L_j\w{f}_p(z,w)-2i\ov{z_j}T\w{f}(\w{z},\w{w})\big)\\
=&(L_j g)(\w{z},\w{w})-2i\ov{z_j}Tg_p({z},{w})
-2i\ov{\w{f}(z_0,w_0)}L_j\w{f}_p(z,w).
\end{split}\end{equation}
Thus
\begin{equation}\begin{split}\label{ljg}
(L_j
g)(\w{z},\w{w})=&L_jg_p(z,w)+2i\ov{\w{f}(z_0,w_0)}L_j\w{f}_p(z,w).
\end{split}\end{equation}
Applying $T$ on this equation, we further get
\begin{equation}\begin{split}\label{ljtg}
(TL_j
g)(\w{z},\w{w})=&TL_jg_p(z,w)+2i\ov{\w{f}(z_0,w_0)}TL_j\w{f}_p(z,w).
\end{split}\end{equation}
Write
\begin{equation}\begin{split}\label{ppqp}
&P_p(z,w)=\left(\begin{array}{cc} L_1{f}_p(z,w)&
\ L_1g_p(z,w)+2i\ov{\w{f}(z_0,w_0)}L_1\w{f}_p(z,w)\\
\vdots &\vdots\\
L_{n-1}{f}_p(z,w)&
\ L_{n-1}g_p(z,w)+2i\ov{\w{f}(z_0,w_0)}L_{n-1}\w{f}_p(z,w)\\
T{f}_p(z,w)& \ Tg_p(z,w)+2i\ov{\w{f}(z_0,w_0)}T\w{f}_p(z,w)
\end{array}\right),\\
& Q_p(z,w)=\left(\begin{array}{c} L_1{\phi}_p(z,w)\\
\vdots \\
L_{n-1}{\phi}_p(z,w)\\
T\phi_p(z,w)
\end{array}\right).
\end{split}\end{equation}

By (\ref{ljftf}), (\ref{ljg}) and (\ref{ppqp}), (\ref{pq}) has the
form
\begin{equation}\label{pqp}
Q_p(z,w)=P_p(z,w)P_p^{-1}(0)Q_p(0).
\end{equation}
\medskip

{\bf Step II. Express the CR Gauss map equation in terms of $F_p^*$}

Recall that $F_p^*$ is defined by
\begin{equation}\begin{split}
F_p^*(z,w)=\frac{1}{\sqrt{\lambda(p)}}F_p(z,w)\left(\begin{array}{rl}
\ov{A^t(p)} & \ 0\\
0 &\ \frac{1}{\sqrt{\lambda(p)}}
\end{array}\right).
\end{split}\end{equation}
Rewrite it as $F_p^*(z,w)=(f_p\ g_p\ \phi_p)(z,w)\cdot M(p)$, then
$M(p)$ takes the following form:
\begin{equation}\begin{split}
&\left(\begin{array}{ccc} M_1&\ 0&\ M_2\\
0 &\frac{1}{\lambda(p)}&0\\
M_3&0 &M_4
\end{array}\right).
\end{split}\end{equation}
Write
\begin{equation}\begin{split}
\hat{M}(p)=\left(\begin{array}{ccc} M_1&\ -\frac{2i}{\lambda}\overline{f^t(z_0,w_0)}&\ M_2\\
0 &\frac{1}{\lambda}&0\\
M_3&-\frac{2i}{\lambda}\overline{\phi^t(z_0,w_0)} &M_4
\end{array}\right).
\end{split}\end{equation}
Since $M(p)$ is independent of $(z,w)$, we have
\begin{equation}\begin{split}
&\big(L_jf_p^{*}\ \ L_jg_p^{*}\ \ L_j\phi_p^{*}\big)(z,w)\\
=&\big(L_jf_p\ \ L_jg_p\ \
L_j\phi_p\big)(z,w)\cdot M(p)\\
=&\big(L_jf_p\ \ L_jg_p+2i\ov{\w{f}(z_0,w_0)}L_j\w{f}_p\ \
L_j\phi_p\big)(z,w)\cdot \hat{M}(p).
\end{split}\end{equation}
Similarly, we get
\begin{equation}\begin{split}
&\big(Tf_p^{*}\ \ Tg_p^{*}\ \ T\phi_p^{*}\big)(z,w) =\big(Tf_p\ \
Tg_p+2i\ov{\w{f}(z_0,w_0)}T\w{f}_p\ \ T\phi_p\big)(z,w)\cdot
\hat{M}(p).
\end{split}\end{equation}
Set
\begin{equation}\begin{split}\label{ppqp*}
P^*_p(z,w)=\left(\begin{array}{cc} L_1{f_p^*}(z,w)&
\ L_1g^*_p(z,w)\\
\vdots &\vdots\\
L_{n-1}{f_p^*}(z,w)&
\ L_{n-1}g^*_p(z,w)\\
T{f^*}_p(z,w)& \ (Tg^*_p)(z,w)
\end{array}\right),
 Q^*_p(z,w)=\left(\begin{array}{c} L_1{\phi_p^*}(z,w)\\
\vdots \\
L_{n-1}{\phi_p^*}(z,w)\\
T\phi_p^*(z,w)
\end{array}\right).
\end{split}\end{equation}

Consider the $n\times N$ matrix $(P^*_p\ Q^*_p)(z,w)$. From
(\ref{ppqp}), (\ref{pqp}) and (\ref{ppqp*}), we yield
\begin{equation}\begin{split}
(P^*_p\ Q^*_p)(z,w)=&(P_p\ Q_p)(z,w)\cdot \hat{M}(p)\\
=&\big(P_p(z,w)\ P_p(z,w)P_p^{-1}(0)Q_p(0)\big) \cdot \hat{M}(p)\\
=&P_p(z,w)P_p^{-1}(0)\cdot\big(P_p(0)\ Q_p(0)\big)\cdot
\hat{M}(p)\\
=&P_p(z,w)P_p^{-1}(0)\cdot\big(P^*_p(0)\ Q^*_p(0)\big).
\end{split}\end{equation}
Hence we know
\begin{equation}\begin{split}
P^*_p(z,w)=P_p(z,w)P_p^{-1}(0)P^*_p(0),\
Q^*_p(z,w)=P_p(z,w)P_p^{-1}(0)Q^*_p(0),
\end{split}\end{equation}
from which we infer
\begin{equation}\begin{split}\label{pqp*}
Q^*_p(z,w)=P^*_p(z,w)\big(P^*_p(0)\big)^{-1}Q^*_p(0).
\end{split}\end{equation}

By the normalization properties of $F_p^*$, we know
$$
P_p^*(0)=\left(\begin{array}{ll} I\ &\ 0\\ a \ & 1
\end{array}\right),\ Q_p^*(0)=\left(\begin{array}{c} 0\\ b
\end{array}\right).
$$
Here we have set $a=(a_1,\cdots,a_{n-1})$ and $b=(b_1,\cdots,b_{N-n})$, where $a_j$ and $b_k$ are defined by (\ref{ab}). Hence (\ref{pqp*}) takes
the form
\begin{equation}\begin{split}\label{pqp*1}
\left(\begin{array}{c} L_1{\phi_p^*}(z,w)\\
\vdots \\
L_{n-1}{\phi_p^*}(z,w)\\
T\phi_p^*(z,w)
\end{array}\right)=P^*_p(z,w)\left(\begin{array}{c} 0\\ b
\end{array}\right)=\left(\begin{array}{c} L_1{g_p^*}(z,w)\\
\vdots \\
L_{n-1}{g_p^*}(z,w)\\
Tg_p^*(z,w)
\end{array}\right)b.
\end{split}\end{equation}

\medskip

{\bf Step III. Express the CR Gauss map equation in terms of $F_p^{**}$}

 We further express this system in the terms of $F_p^{**}$,
which is defined as follows:
\begin{equation}\begin{split}\label{fp**}
&\w{f}^{**}_p=\frac{1}{q^*(z,w)}\big(\w{f^*_p}(z,w)-ag^*_p(z,w)\big),\
\w{g}^{**}_p=\frac{1}{q^*(z,w)}g^*_p.
\end{split}\end{equation}
Here we have set
\begin{equation}\begin{split}\label{q*}
q^*(z,w)&=1+2i\ov{c}\w{f^*_p}+(r-i|c|^2)g^*_p(z,w),\ c=\frac{\p
\w{f}^*_p}{\p w}(0)=(a,b),\ r=\frac{1}{2}\text{Re}\big(\frac{\p ^2
g_p^*}{\p w^2}(0)\big).
\end{split}\end{equation}
Write $q^{**}(z,w)=\frac{1}{q^*(z,w)}$, then (\ref{fp**}) and
(\ref{q*}) gives
\begin{equation*}\begin{split}
q^{**}(z,w)&=1-2i\ov{c}\frac{\w{f}^*_p}{q^*}(z,w)-(r-i|c|^2)\frac{g^*_p}{q^*}(z,w)\\
&=1-2i\ov{c}{f^{**}_p}(z,w)-(r+i|c|^2)g_p^{**}(z,w).
\end{split}\end{equation*}
Together with (\ref{fp**}), we know
\begin{equation}\begin{split}\label{fp**1}
&\w{f}^{*}_p=\frac{1}{q^{**}(z,w)}\big(\w{f^{**}_p}(z,w)+cg^{**}_p(z,w)\big),\
{g}^{*}_p=\frac{1}{q^{**}(z,w)}g^{**}_p.
\end{split}\end{equation}
Applying $L_j$ and $T$ to (\ref{fp**1}), we yield
\begin{equation*}\begin{split}
L_j\phi^{*}_p=&\frac{1}{(q^{**}(z,w))^2}\big(L_j(\phi^{**}_p+bg^{**}_p)q^{**}
-(\phi^{**}_p+bg^{**}_p)L_jq^{**}\big),\\
T\phi^{*}_p=&\frac{1}{(q^{**}(z,w))^2}\big(T(\phi^{**}_p+bg^{**}_p)q^{**}
-(\phi^{**}_p+bg^{**}_p)Tq^{**}\big),\\
L_jg^{*}_p=&\frac{1}{(q^{**}(z,w))^2}\big(L_jg^{**}_pq^{**}
-g^{**}_pL_jq^{**}\big),\\
Tg^{*}_p=&\frac{1}{(q^{**}(z,w))^2}\big(Tg^{**}_pq^{**}
-g^{**}_pTq^{**}\big).
\end{split}\end{equation*}
Substituting these relations into (\ref{pqp*1}), we get
\begin{equation*}\begin{split}
&L_j(\phi^{**}_p+bg^{**}_p)q^{**}
-(\phi^{**}_p+bg^{**}_p)L_jq^{**}=b\big(L_jg^{**}_pq^{**}
-g^{**}_pL_jq^{**}\big),\\
&T(\phi^{**}_p+bg^{**}_p)q^{**}
-(\phi^{**}_p+bg^{**}_p)Tq^{**}=b\big(Tg^{**}_pq^{**}
-g^{**}_pTq^{**}\big).
\end{split}\end{equation*}
A quick simplification gives
\begin{equation}\begin{split}\label{115}
&L_j\phi^{**}_pq^{**}-\phi^{**}_pL_jq^{**}=0,\ T\phi^{**}_pq^{**}
-\phi^{**}_pTq^{**}=0.
\end{split}\end{equation}
Notice that $\phi^{**}_p=O(|(z,w)|^2)$ and $q^{**}_p=1+O(|(z,w)|)$. Hence (\ref{115}) takes the form
\begin{equation}\begin{split}\label{pqp**}
&\frac{\p}{\p z_j}\phi^{**}_p=O(|(z,w)|^2),\ \frac{\p}{\p
w}\phi^{**}_p=O(|(z,w)|^2).
\end{split}\end{equation}
\medskip

{\bf Step IV. Express the CR Gauss map equation in terms of $F_p^{***}$}

Finally, we would like to express the system (\ref{pqp**}) in terms
of the map $F_P^{***}$ defined by (\ref{fp***})
Then (\ref{pqp**}) takes the following form:
\begin{equation}\begin{split}\label{pqp***}
&\big(\frac{\p \phi^{***}_p}{\p
z_j}\big)\big(zU,w\big)=O(|(z,w)|^2),\ \big(\frac{\p
\phi^{***}_p}{\p w}\big)\big(zU,w\big)=O(|(z,w)|^2).
\end{split}\end{equation}

{\bf Step V. Express the CR Gauss map equation in terms of $F_p^{****}$}

Let $\tau\in \text{Aut}_0(\mathbb{H}_n)$ and $\sigma\in
\text{Aut}_0(\mathbb{H}_N)$ be given by (\ref{st}), write
$F_p^{****}=\tau \circ F_p^{***} \circ \sigma$. Then
\begin{equation}\begin{split}
&\phi_p^{****}(z,w)=\phi_p^{***}    \big(\sigma(z,w)\big)+O(|(z,w)|^3),\\
&\sigma(z,w)=(z-cw,w)+O(|(z,w)|^2).
\end{split}\end{equation}
Hence (\ref{pqp***}) takes the following form:
\begin{equation}\begin{split}
&\frac{\p \phi^{****}_p}{\p z_j} =O(|(z,w)|^2),\ \  \frac{\p \phi^{****}_p}{\p w} =O(|(z,w)|^2).
\end{split}\end{equation}
This finishes the proof of the theorem.

\section{Some normalization properties}

In this section, we will derive some properties for proper
holomorphic maps between balls.

Let $F:\p \mathbb{H}_n\rightarrow \p \mathbb{H}_N$ be a rational CR
immersion.  Assume the geometric rank of $F$ at $0$  $\kappa_0\in [1,n-2]$
 is maximal.
 We also suppose that $F$ has the
following expansion near $0$:
\begin{equation}\begin{cases}
f_{j}=z_j+\frac{i}{2}\mu_jz_jw+d_jw^2+O(|(z,w)|^3)\ \text{for}\
j\leq \kappa_0,\\
f_{j}=z_j+d_jw^2+O(|(z,w)|^3)\ \text{for}\
\kappa_0<j\leq n-1,\\
\phi_{jk}=\mu_{jk}z_jz_k+\sum\limits_{h=1}^{n-1}e_{h,jk}z_hw+d_{jk}w^2+O(|(z,w)|^3)\
\text{for}\
(j,k)\in \mathcal{S}_0,\\
\phi_{jk}=\sum\limits_{h=1}^{n-1}e_{h,jk}z_hw+d_{jk}w^2+O(|(z,w)|^3)\
\text{for}\
(j,k)\in \mathcal{S}_1,\\
\ g=w+O(|(z,w)|^3).
\end{cases}\end{equation}
Let $\tau\in \text{Aut}_0(\mathbb{H}_n)$ and $\sigma\in
\text{Aut}_0(\mathbb{H}_n)$ be given by (\ref{st}).

\begin{lem}\label{hatexp}
Let $\hat{F}=(\hat{f},\hat{\phi},g):=\tau \circ F \circ \sigma$, where $\tau$ and $\sigma$ are given in (\ref{st}),
then $\hat{F}$ has the following expansion:
\begin{equation}\begin{cases}
\hat{f}_{j}=z_j+\frac{i}{2}\mu_jz_jw+(d_j-\frac{i}{2}\mu_jc_j)w^2+O(|(z,w)|^3)\
\text{for}\
j\leq \kappa_0,\\
\hat{f}_{j}=z_j+d_jw^2+O(|(z,w)|^3)\ \text{for}\
\kappa_0<j\leq n-1,\\
\hat{\phi}_{jk}=\mu_{jk}(z_j-c_jw)(z_k-c_kw)+\sum\limits_{h=1}^{n-1}e_{h,jk}(z_h-c_hw)w+d_{jk}w^2\\
\hskip 1cm+O(|(z,w)|^3)\ \text{for}\
(j,k)\in \mathcal{S}_0,\\
\hat{\phi}_{jk}=\sum\limits_{h=1}^{n-1}e_{h,jk}(z_h-c_hw)w+d_{jk}w^2+O(|(z,w)|^3)\
\text{for}\
(j,k)\in \mathcal{S}_1,\\
\ \hat{g}=w+O(|(z,w)|^3).
\end{cases}\end{equation}
\end{lem}

\begin{proof} A direct computation from (\ref{st})  shows that
\begin{equation}\begin{split}
\sigma(z,w)=&\Big(\big((z_{j}-c_{j}w\big)\cdot\big(1-2i\langle
\ov{c},z \rangle+i|c|^2w)\big)_{1\leq j\leq n-1},w\big(1-2i\langle
\ov{c},z \rangle+i|c|^2w\big)\Big)\\
&+O(|(z,w)|^3).
\end{split}\end{equation}
For $1\leq j\leq \kappa_0$ and $\kappa_0+1\leq k\leq n-1$, we have
\begin{equation}\begin{split}
f_j\circ\sigma(z,w)=&\big(z_{j}-c_{j}w\big)\cdot\big(1-2i\langle
\ov{c},z
\rangle+i|c|^2w\big)+\frac{i}{2}\mu_j(z_j-c_jw)w\\
&+d_jw^2+O(|(z,w)|^3),\\
f_k\circ\sigma(z,w)=&\big(z_{k}-c_{k}w\big)\cdot\big(1-2i\langle
\ov{c},z \rangle+i|c|^2w\big)+d_kw^2+O(|(z,w)|^3).
\end{split}\end{equation}
Similarly, for $(j,l)\in \mathcal{S}_0$ and $(j',l')\in
\mathcal{S}_1$, we get
\begin{equation}\begin{split}
\phi_{jl}\circ\sigma(z,w)&=\mu_{jl}(z_j-c_jw)(z_l-c_lw)+\sum\limits_{h=1}^{n-1}e_{h,jl}(z_h-c_hw)w
+d_{jl}w^2+O(|(z,w)|^3),\\
\phi_{j'l'}\circ\sigma(z,w)&=\sum\limits_{h=1}^{n-1}e_{h,j'l'}(z_h-c_hw)w+d_{j'l'}w^2+O(|(z,w)|^3),\\
g\circ\sigma(z,w)&=w\big(1-2i\langle \ov{c},z
\rangle+i|c|^2w\big)+O(|(z,w)|^3).
\end{split}\end{equation}
Hence for $1\leq j\leq \kappa_0$ and $\kappa_0+1\leq k\leq n-1$, we
further get
\begin{equation}\begin{split}
f_j\circ\sigma(z,w)+c_jg\circ\sigma(z,w)=&z_{j}\cdot\big(1-2i\langle
\ov{c},z
\rangle+i|c|^2w\big)+\frac{i}{2}\mu_j(z_j-c_jw)w\\
&+d_jw^2+O(|(z,w)|^3),\\
f_k\circ\sigma(z,w)+c_kg\circ\sigma(z,w)=&z_{k}\cdot\big(1-2i\langle
\ov{c},z \rangle+i|c|^2w\big)+d_kw^2+O(|(z,w)|^3).
\end{split}\end{equation}

Substituting  the formulas above into (\ref{st}), we yield for
$1\leq j\leq \kappa_0$ that
\begin{equation}\begin{split}
\hat{f}_j=&\Big(z_{j}\cdot\big(1-2i\langle \ov{c},z
\rangle+i|c|^2w\big)+\frac{i}{2}\mu_j(z_j-c_jw)w+d_jw^2\Big)\\
&\cdot \big(1+2i\langle \ov{c},z-cw
\rangle+i|c|^2w\big) +O(|(z,w)|^3)\\
=&z_j+\frac{i}{2}\mu_jz_jw+(d_j-\frac{i}{2}\mu_jc_j)w^2+O(|(z,w)|^3).
\end{split}\end{equation}
Similarly, we obtain for $\kappa_0+1\leq k\leq n-1$ that
\begin{equation}\begin{split}
\hat{f}_k =&z_k+d_kw^2+O(|(z,w)|^3).
\end{split}\end{equation}
For $\phi$, we have
\begin{equation}\begin{split}
\hat{\phi}_{jl}=&\mu_{jl}(z_j-c_jw)(z_l-c_lw)+\sum\limits_{h=1}^{\kappa_0}e_{h,jl}(z_h-c_hw)w\\
&+d_{jl}w^2+O(|(z,w)|^3) \
\text{for}\ (j,l)\in \mathcal{S}_0,\\
\hat{\phi}_{jl}=&\sum\limits_{h=1}^{\kappa_0}e_{h,jl}(z_h-c_hw)w+d_{jl}w^2+O(|(z,w)|^3)\
\text{for}\ (j,l)\in \mathcal{S}_1.
\end{split}\end{equation}
And for $g$, we have
\begin{equation}\begin{split}
\hat{g}=&w\big(1-2i\langle \ov{c},z \rangle+i|c|^2w\big)\cdot
(1+2i\langle \ov{c},z-cw \rangle+i|c|^2w\big) +O(|(z,w)|^3)\\
=&w+O(|(z,w)|^3).
\end{split}\end{equation}
The proof of Lemma \ref{hatexp} is complete.
\end{proof}

As a consequence of the lemma, we can further normalize the map
(\ref{eqn:hao}) such that
\begin{equation}\begin{split}\label{11j}
e_{1,1\alpha}=0\ \text{for}\ \alpha>\kappa_0.
\end{split}\end{equation}

Indeed, if we choose $c_j=0$ for $1\leq j\leq \kappa_0$, then
$$
\hat{\phi}_{j\al}^{(I_j+I_n)}=e_{j,j\al}-\sqrt{\mu_j}c_{\al}
\text{for}\ 1\leq j\leq \kappa_0<\al\leq n-1.
$$
Thus we can choose $c_{\al}$ for $\kappa_0+1\leq \al\leq n-1$ such
that $\hat{\phi}_{1\al}^{(I_1+I_n)}=0$ with $\kappa_0< \al\leq n-1$.
By \cite{HJX06}, the map has the following form:
\begin{equation}\begin{cases}\label{fp****}
f^{(****)}_{p,j}=z_j+\frac{i}{2}\mu_{p,j}z_jw+O(|(z,w)|^3)\
\text{for}\
j\leq \kappa_0,\\
f^{(****)}_{p,j}=z_j\ \text{for}\
\kappa_0<j\leq n-1,\\
\phi^{(****)}_{p,jk}=\mu_{p,jk}z_jz_k+\sum\limits_{h=1}^{\kappa_0}e_{ph,jk}z_hw+O(|(z,w)|^3)\
\text{for}\
(j,k)\in \mathcal{S}_0,\\
\phi^{(****)}_{p,jk}=\sum\limits_{h=1}^{n-1}e_{ph,jk}z_hw+O(|(z,w)|^3)\
\text{for}\
(j,k)\in \mathcal{S}_1,\\
\ g^{(****)}_p=w.
\end{cases}\end{equation}
Here $e_{p1,1\al}=0$ for $\kappa_0+1\leq \al\leq n-1$. Write
$$\Phi^{****(1,1)}_{p}=\big(\sum\limits_{h=1}^{\kappa_0}e_{ph,jk}\big)_{(j,k)\in
\mathcal{S}}.$$

Next, we recall some relations derived by analyzing the Chern-Moser
equation. The following relations are obtained in \cite{HJY14} for
geometric two case, which in fact holds true independent of the
geometric rank and the codimension of the maps.

As in \cite{HJY14}, write $\xi_j(z) = \overline{{e}_j} \cdot
\Phi^{(2,0)}_0(z)$.
 From \cite[(3.5)]{HJY14}, we know
\begin{equation}\label{hh}
\ov z f^{(2,1)}(z) = - \ov{\text{{$(z_1, ..., z_{\kappa_0})$}}}\cdot
\xi(z).
\end{equation}
By a similar argument as that of \cite[(4.3)]{HJY14}, we obtain
\begin{equation}\begin{split}\label{43}
&\phi^{(1,2)}\in Span\{e_1,\cdots,e_{\kappa_0}\}.
\end{split}\end{equation}
 From
\cite[(4.10)]{HJY14}, we know
\begin{equation}\begin{split}\label{112}
&2\text{Re}\big(\ov{z}f^{(1,2)}(z)\big)+|f^{(1,1)}(z)|^2+|\phi^{(1,1)}(z)|^2=0.
\end{split}\end{equation}

With these preparations, we come to the following main theorem of
this section.

\begin{thm}\label{phi110}
  Let $F:\p \mathbb{H}_n \rightarrow \p \mathbb{H}_N$ be a CR
  immersion as in Theorem \ref{main}.
  For generic $p$ around $0$, the
  associated vector $\Phi^{****(1,1)}_{p}(z)\not\equiv 0$.
\end{thm}
\begin{proof}
By Theorem \ref{thm3}, we can suppose that near $0$, the expansion
has the form (\ref{eqn:hao}). Since $\Phi^{****(1,1)}_{p}(z)$ is a smooth function with respect to $p$ and $\Phi^{****(1,1)}_{0}(z) = \Phi^{(1,1)}(z)$ by notation, 
it suffices to prove  $\Phi^{(1,1)}(z)\neq 0$. Assume by
contradiction that $\Phi^{(1,1)}(z)\equiv 0$. By our notation, this
implies $e=\xi=0$. From (\ref{hh})-(\ref{43}), we know
 \begin{equation}\label{zero}
    f^{(2,1)}=\phi^{(1,2)}=0.
 \end{equation}

\medskip

Next we would like to give some asymptotic properties of the
coefficients $F_p^{****}$. For $1\leq j\leq \kappa_0$, we have
\begin{equation}\begin{split}\label{ej}
E_j(p)=&\big(\frac{\p \w{f}_p}{\p
z_j}\big)\big|_0=L_j(\w{f})(p)=\big((\frac{\p}{\p
z_j}+2i\overline{z_{0j}}\frac{\p}{\p w})\w{f}\big)(p)\\
:=&\big(E^{[0]}_j(p),E^{[1]}_j(p),\cdots,E^{[k_0+1]}_j(p)\big).
\end{split}\end{equation}
Here we have set
\begin{equation*}\begin{split}
E^{[0]}_j(p)&=L_j(f)(p)\in \mathbb{C}^{n-1},\\
E^{[k]}_j(p)&=\big(L_j(\phi_{kk})(p),L_j(\phi_{k(k+1)})(p),\cdots,L_j(\phi_{k(n-1)})(p)\big)\in
\mathbb{C}^{n-k}.
\end{split}\end{equation*}
Then
\begin{equation}\begin{split}\label{n1}
E^{[0]}_j(p)&=\big((\frac{\p}{\p
z_j}+2i\overline{z_{0j}}\frac{\p}{\p w}){f}\big)(p)\\
&=\Big(0,\cdots,0,1+\frac{i}{2}\mu_ju_0\ (j\text{-th
position}),0,\cdots,0\Big)+O(2).
\end{split}\end{equation}
and
\begin{equation}\label{n2}
E^{[k]}_j(p)=\left\{\begin{array}{ll}
(0,\cdots,0,\mu_{kj}z_{0k},0,\cdots,0)+O(2),\ &k<j,\\
(2\mu_{kk}z_{0k},\mu_{k(k+1)}z_{0(k+1)},\cdots,\mu_{k(n-1)}z_{0(n-1)})+O(2),\ &k=j,\\
O(2), &k>j.
\end{array}\right.
\end{equation}

By the definition of $C_{ij}$ (see \cite[p.17]{Hu99}), the
asymptotic expansion of $E_j(p)$ given in (\ref{n1}) and (\ref{n2}),
we can suppose that $C_{jk}$ has the following form
$$
C_{jk}=\big(o(1),\cdots,o(1),1+o(1),0,\cdots,0\big),
$$
where $1+o(1)$ is in the position corresponding to that of
$\phi_{jk}$ in $\w{f}$. Since $A$ defined by (\ref{ap}) is unitary,
we can use the implicit function theory to get that
\begin{equation}
C_{jk}=\left\{\begin{array}{ll}
(0,\cdots,0,-\mu_{jk}\ov{z_{ok}},0,\cdots,0,-\mu_{jk}\ov{z_{oj}},0,\cdots,0,1,0,\cdots,0)+O(2)
\ &j<k, j\leq \kappa_0,\\
(0,\cdots,0,-2\mu_{jj}\ov{z_{oj}},0,\cdots,0,1,0,\cdots,0)+O(2)
\ &j=k,\\
(0,\cdots,0,1,0,\cdots,0)+O(2)
\ &j= \kappa_0+1,\\
\end{array}\right.
\end{equation}

From (\ref{ljtf}) and $f^{(2,1)}(z)\equiv0$, we know for $1\leq
j\leq \kappa_0$ the following
\begin{equation}\begin{split}
f_{p,j}^{(1,1)}=&\sum\limits_{k=1}^{n-1}L_kTf_j(p)z_k=
\frac{i}{2}\mu_jz_j+\sum\limits_{k=1}^{\kappa_0}2f_j^{(I_k+2I_n)}u_0z_k+O(|(z_0,u_0)|^2)z.
\end{split}\end{equation}
For $\kappa_0+1\leq j\leq n-1$, we have
\begin{equation}\begin{split}
f_{p,j}^{(1,1)}=&\sum\limits_{k=1}^{n-1}L_kTf_j(p)z_k =0.
\end{split}\end{equation}
For $1\leq j\leq n-1$, we have
\begin{equation}\begin{split}
f_{p,j}^{(0,2)}=&\frac{1}{2}T^2f_{p,j}(0)=\frac{1}{2}T^2f_j(p)=
\sum\limits_{k=1}^{\kappa_0}f_j^{(I_k+2I_n)}z_{0k}+O(|(z_0,u_0)|^2).
\end{split}\end{equation}

From (\ref{ljtf}), we know
\begin{equation}\begin{split}
&L_jT\phi_p(0)=\frac{\p^2}{\p z_j \p
w}\phi(p)+2i\ov{z_{0j}}\frac{\p^2}{\p w^2}\phi(p),\
T^2\phi_p(0)=\frac{1}{2}\frac{\p^2}{ \p w^2}\phi(p).
\end{split}\end{equation}
Thus we get
\begin{equation}\begin{split}
\phi_{p,jl}^{(1,1)}=&\sum\limits_{k=1}^{\kappa_0}2\phi_{jl}^{(2I_k+I_n)}z_{0k}z_k+\sum\limits_{k\neq
h,k,h=1}^{n-1}\phi_{jl}^{(I_h+I_k+I_n)}z_{0h}z_k+O(|(z_0,u_0)|^2)z
\end{split}\end{equation}
and
\begin{equation}\begin{split}
\phi_{p,jl}^{(0,2)}=&\phi_{jl}^{(1,2)}(z_0)+O(|(z_0,u_0)|^2)=O(|(z_0,u_0)|^2).
\end{split}\end{equation}
 Here the last equality follows from (\ref{zero}). From (\ref{ljtg}), we know
\begin{equation}\begin{split}
g_p^{(1,1)}&=\sum\big(L_jTg(p)-2i\ov{\w{f}(p)}\cdot
L_jT\w{f}(p)\big)z_j=O(|(z_0,u_0)|)z,\\
g_p^{(0,2)}&=O(|(z_0,u_0)|).
\end{split}\end{equation}
Notice that $\lambda(p)=1+O(2)$, thus for $1\leq j\leq \kappa_0$
\begin{equation}\begin{split}
\phi_{p,jj}^{*(1,1)}=&\big(\frac{1}{\sqrt{\lambda(p)}}\w{f}_p\cdot
\ov{C^t_{jj}}\big)^{(1,1)}\\
=&f_{p,j}^{(1,1)}\cdot
(-2\mu_{jj}{z_{0j}})+\phi_{p,jj}^{(1,1)}+O(|(z_0,u_0)|^2)z\\
=&-i\mu_j\mu_{jj}{z_{0j}}z_j+\sum\limits_{k=1}^{\kappa_0}2\phi_{jj}^{(2I_k+I_n)}z_{0k}z_k+\sum\limits_{k\neq
h,k,h=1}^{n-1}\phi_{jj}^{(I_h+I_k+I_n)}z_{0h}z_k\\
&+O(|(z_0,u_0)|^2)z.
\end{split}\end{equation}
Similarly, we have
\begin{equation}\begin{split}
f_{p,j}^{*(0,2)}=&\big(\frac{1}{\lambda(p)}\w{f}_p\cdot
\ov{E^t_{j}}\big)^{(0,2)}
=\sum\limits_{k=1}^{\kappa_0}f_j^{(I_k+2I_n)}z_{0k}+O(|(z_0,u_0)|^2),\\
\phi_{p,jj}^{*(0,2)}=&O(|(z_0,u_0)|^2).
\end{split}\end{equation}

Notice that
\begin{equation*}\begin{split}
a,r=o(1),\ U=I+o(1),\  U^*=I+o(1),\ \phi_{p,jk}^{*(1,1)}=o(1),\
f_{p,j}^{*(0,2)}=o(1), \ \phi_{p,jk}^{*(0,2)}=o(1).
\end{split}\end{equation*}
A straightforward computation shows that
\begin{equation}\begin{split}\label{phi****}
\phi_{p,jj}^{***(1,1)}=&-i\mu_j\mu_{jj}{z_{0j}}z_j+\sum\limits_{k=1}^{\kappa_0}2\phi_{jj}^{(2I_k+I_n)}
z_{0k}z_k+\sum\limits_{k\neq
h,k,h=1}^{n-1}\phi_{jj}^{(I_h+I_k+I_n)}z_{0h}z_k\\
&+O(|(z_0,u_0)|^2)z,\\
f_{p,j}^{***(0,2)}=&\sum\limits_{k=1}^{\kappa_0}f_j^{(I_k+2I_n)}z_{0k}+O(|(z_0,u_0)|^2),\\
\phi_{p,jj}^{***(0,2)}=&O(|(z_0,u_0)|^2).
\end{split}\end{equation}
By Lemma \ref{hatexp}, to make $f_p^{****(0,2)}=0$, we have to
choose $c_j$ for $1\leq j\leq \kappa_0$ such that
\begin{equation}\begin{split}\label{cj}
c_j=-\frac{2i}{\mu_j}f_{p,j}^{***(0,2)}
=-\frac{2i}{\mu_j}\sum\limits_{k=1}^{\kappa_0}f_j^{(I_k+2I_n)}z_{0k}+O(|(z_0,u_0)|^2).
\end{split}\end{equation}
From  Lemma \ref{hatexp} and (\ref{phi****}), we know
\begin{equation}\begin{split}
\phi_{p,jj}^{****(1,1)}(z)=&-i\mu_j\mu_{jj}{z_{0j}}z_j+\sum\limits_{k=1}^{\kappa_0}2\phi_{jj}^{(2I_k+I_n)}
z_{0k}z_k+\sum\limits_{k\neq
h,k,h=1}^{n-1}\phi_{jj}^{(I_h+I_k+I_n)}z_{0h}z_k\\
&-2\sqrt{\mu_j}c_jz_{j}+O(|(z_0,u_0)|^2)z.
\end{split}\end{equation}

Thus the coefficients of $z_k$ term in $\phi_{p,jj}^{****(1,1)}(z)$
is the following:
\begin{equation}\begin{split}
I_j(z_0,u_0):=&-i\mu_j\mu_{jj}{z_{0j}}+2\phi_{jj}^{(2I_j+I_n)}
z_{0j}+\sum\limits_{ k\neq
j,k=1}^{n-1}\phi_{jj}^{(I_k+I_j+I_n)}z_{0k}+\frac{4i}{\sqrt{\mu_j}}\sum\limits_{k=1}^{\kappa_0}
f_j^{(I_k+2I_n)}z_{0k}\\
&+O(|(z_0,u_0)|^2),\\
I_k(z_0,u_0):=&2\phi_{jj}^{(2I_k+I_n)} z_{0k}+\sum\limits_{h\neq
k}\phi_{jj}^{(I_h+I_k+I_n)}z_{0h}+O(|(z_0,u_0)|^2)\ \text{for}\
k\neq j.
\end{split}\end{equation}
If $\phi_{p,jj}^{****(1,1)}(z)\equiv 0$ in a neighborhood of
$0$, then $I_h(z_0,u_0)\equiv 0$ in a a neighborhood of $0$ for
$1\leq h\leq n-1$. From $I_k(z_0,u_0)\equiv 0$ for $k\neq j$, we
know $\phi_{jj}^{(I_k+I_h+I_n)}=0$ for $k\neq j$ and $1\leq j\leq
n-1$. From $I_j(z_0,u_0)\equiv 0$, we get
\begin{equation}\begin{split}\label{eq11}
&f_j^{(I_k+2I_n)}=0\ \text{for}\ k\neq j,\\
&-i\mu_j\mu_{jj}+2\phi_{jj}^{(2I_j+I_n)}+\frac{4i}{\sqrt{\mu_j}}f_j^{(I_j+2I_n)}=0.
\end{split}\end{equation}
Thus
\begin{equation}\begin{split}\label{cc}
&\phi_{p,jj}^{***(1,1)}=\Big(-i\mu_j\mu_{jj}+2\phi_{jj}^{(2I_j+I_n)}\Big)z_{0j}z_j,\\
&c_j=-\frac{2i}{\mu_j}f_j^{(I_j+2I_n)}z_{0j}+O(|(z_0,u_0)|^2).
\end{split}\end{equation}

First, by (4.34), we know
\begin{equation}\begin{split}\label{eq1111}
&-i\mu_j\mu_{jj}+2\phi_{jj}^{(2I_j+I_n)}+\frac{4i}{\sqrt{\mu_j}}f_j^{(I_j+2I_n)}=0.
\end{split}\end{equation}

For the pair $(j,l)$ satisfying $1\leq j\leq l\leq \kappa_0$, by the
asymptotic expansion of $C_{jl}$, we get
\begin{equation}\begin{split}
\phi_{p,jl}^{*(1,1)}=&\big(\frac{1}{\sqrt{\lambda(p)}}\w{f}_p\cdot
\ov{C^t_{jl}}\big)^{(1,1)}\\
=&f_{p,j}^{(1,1)}\cdot (-\mu_{jl}{z_{0l}})+f_{p,l}^{(1,1)}\cdot
(-\mu_{jl}{z_{0j}})+\phi_{p,jl}^{(1,1)}+O(|(z_0,u_0)|^2)z\\
=&-\frac{i}{2}\mu_j\mu_{jl}{z_{0l}}z_j-\frac{i}{2}\mu_l\mu_{jl}{z_{0j}}z_l
+\sum\limits_{k=1}^{\kappa_0}2\phi_{jl}^{(2I_k+I_n)}z_{0k}z_k\\
&+\sum\limits_{k\neq
h,k,h=1}^{n-1}\phi_{jl}^{(I_h+I_k+I_n)}z_{0h}z_k+O(|(z_0,u_0)|^2)z.
\end{split}\end{equation}
We further know
\begin{equation}\begin{split}
\phi_{p,jl}^{****(1,1)}(z)=&-\frac{i}{2}\mu_j\mu_{jl}{z_{0l}}z_j-\frac{i}{2}\mu_l\mu_{jl}{z_{0j}}z_l
+\sum\limits_{k=1}^{\kappa_0}2\phi_{jl}^{(2I_k+I_n)}z_{0k}z_k\\
&+\sum\limits_{k\neq
h,k,h=1}^{n-1}\phi_{jl}^{(I_h+I_k+I_n)}z_{0h}z_k-\mu_{jl}c_lz_{j}-\mu_{jl}c_jz_{l}+O(|(z_0,u_0)|^2)z.
\end{split}\end{equation}
By comparing the coefficients of $z_k$ term in
$\phi_{p,jl}^{****(1,1)}(z)$, we obtain
\begin{equation}\begin{split}
&\phi_{jl}^{(I_h+I_k+I_n)}=0\ \text{for}\ h<l,\ (h,k)\neq
(j,j),(j,l),(l,j),(l,l).\\
&-\frac{i}{2}\mu_j\mu_{jl}{z_{0l}}+2\phi_{jl}^{(2I_j+I_n)}z_{0j}+\sum\limits_{h\neq
j,h=1}^{n-1}\phi_{jl}^{(I_j+I_h+I_n)}z_{0h}-\mu_{jl}c_l=0,\\
&-\frac{i}{2}\mu_l\mu_{jl}{z_{0j}}+2\phi_{jl}^{(2I_l+I_n)}z_{0l}+\sum\limits_{h\neq
l,h=1}^{n-1}\phi_{jl}^{(I_l+I_h+I_n)}z_{0h}-\mu_{jl}c_j=0.
\end{split}\end{equation}
By the formula for $c_j$ in (\ref{cc}), we know $\phi_{jl}^{(I_k+I_h+I_n)}=0$ for
$(k,h)\neq (j,l),(l,j)$ and
\begin{equation}\begin{split}\label{phi12}
&-\frac{i}{2}\mu_j\mu_{jl}+\phi_{jl}^{(I_j+I_l+I_n)}+\frac{2i\mu_{jl}}{\mu_l}f_l^{(I_l+2I_n)}=0,\\
&-\frac{i}{2}\mu_l\mu_{jl}+\phi_{jl}^{(I_j+I_l+I_n)}+\frac{2i\mu_{jl}}{\mu_j}f_j^{(I_j+2I_n)}=0.
\end{split}\end{equation}
Eliminating $\phi_{jl}^{(I_j+I_l+I_n)}$ in the above system, we get
\begin{equation}\begin{split}\label{cc1}
&\frac{i}{2}\mu_j+\frac{2i}{\mu_j}f_j^{(I_j+2I_n)}=\frac{i}{2}\mu_l+\frac{2i}{\mu_l}f_l^{(I_l+2I_n)}.
\end{split}\end{equation}
 Write
\begin{equation}\begin{split}
A_i=f_i^{(I_i+2I_n)},\ B_i=\phi_{ii}^{(2I_i+I_n)}.
\end{split}\end{equation}
Then by (\ref{eq11}) and (\ref{cc1}),
\begin{equation}\begin{split}\label{aj}
&-i\mu_j\mu_{jj}+2B_j+\frac{4i}{\sqrt{\mu_j}}A_j=0.\\
&\frac{i}{2}\mu_j+\frac{2i}{\mu_j}A_j=\frac{i}{2}\mu_l+\frac{2i}{\mu_l}A_l.
\end{split}\end{equation}
From (\ref{phi12}), we get
\begin{equation}\begin{split}\label{phi1200}
&\phi_{jl}^{(I_j+I_l+I_n)}=\frac{i}{2}\mu_j\mu_{jl}-\frac{2i\mu_{jl}}{\mu_l}A_l.
\end{split}\end{equation}

By the asymptotic expansion of  $C_{j\al}$ for $1\leq j\leq
\kappa_0<\alpha\leq n-1$, we get
\begin{equation}\begin{split}
\phi_{p,j\al}^{*(1,1)}=&\big(\frac{1}{\sqrt{\lambda(p)}}\w{f}_p\cdot
\ov{C^t_{j\al}}\big)^{(1,1)}\\
=&f_{p,j}^{(1,1)}\cdot (-\mu_{j\al}{z_{0\al}})+\phi_{p,j\al}^{(1,1)}+O(|(z_0,u_0)|^2)z\\
=&-\frac{i}{2}\mu_j\mu_{j\al}{z_{0\al}}z_j
+\sum\limits_{k=1}^{\kappa_0}2\phi_{j\al}^{(2I_k+I_n)}z_{0k}z_k\\
&+\sum\limits_{k\neq
h,k,h=1}^{n-1}\phi_{j\al}^{(I_h+I_k+I_n)}z_{0h}z_k+O(|(z_0,u_0)|^2)z.
\end{split}\end{equation}
We further know
\begin{equation}\begin{split}
\phi_{p,j\al}^{****(1,1)}(z)=&-\frac{i}{2}\mu_j\mu_{j\al}{z_{0\al}}z_j
+\sum\limits_{k=1}^{\kappa_0}2\phi_{j\al}^{(2I_k+I_n)}z_{0k}z_k\\
&+\sum\limits_{k\neq
h,k,h=1}^{n-1}\phi_{j\al}^{(I_h+I_k+I_n)}z_{0h}z_k-\mu_{j\al}c_{\al}
z_{j}-\mu_{j\al}c_jz_{\al}+O(|(z_0,u_0)|^2)z.
\end{split}\end{equation}
By comparing the coefficients of $z_k$ term in
$\phi_{p,j\al}^{****(1,1)}(z)$, we obtain
\begin{equation}\begin{split}
&\phi_{j\al}^{(I_h+I_k+I_n)}=0\ \text{for}\ (h,k)\neq
(j,j),(j,\al),(\al,j).\\
&-\frac{i}{2}\mu_j\mu_{j\al}{z_{0\al}}+2\phi_{j\al}^{(2I_j+I_n)}z_{0j}+\sum\limits_{h\neq
j,h=1}^{n-1}\phi_{j\al}^{(I_j+I_h+I_n)}z_{0h}-\mu_{j\al}c_\al=0,\\
&\sum\limits_{h\neq
\alpha,h=1}^{n-1}\phi_{j\al}^{(I_\alpha+I_h+I_n)}z_{0h}-\mu_{j\al}c_j=0.
\end{split}\end{equation}
By the formula for $c_j$, we know $\phi_{j\al}^{(I_h+I_k+I_n)}=0$
for $(h,k)\neq (j,\al),(\al,j)$ and
\begin{equation}\begin{split}
&-\frac{i}{2}\mu_j\mu_{j\al}z_{0\al}+\phi_{j\al}^{(I_j+I_\al+I_n)}z_{0\al}=\mu_{j\al}c_{\al},\\
&\phi_{j\al}^{(I_j+I_\al+I_n)}-\mu_{j\al}c_j=0.
\end{split}\end{equation}
Hence
\begin{equation}\begin{split}
\phi_{j\al}^{(I_j+I_\al+I_n)}=\mu_{j\al}(-\frac{2i}{\mu_1})f_j^{(I_j+2I_n)}=-\frac{2i}{\sqrt{\mu_j}}A_j.
\end{split}\end{equation}

By considering $\phi_{p,\al\be}^{****(1,1)}(z)$ for $(\al,\be)\in
\mathcal{S}_1$, we know $\phi_{\al\be}^{(2,1)}=0$.

\bigskip

By [HJY,(4.10)], we have
\begin{equation*}\begin{split}
&2\text{Re}\big(\ov{z}f^{(1,2)}(z)\big)+|f^{(1,1)}(z)|^2+|\phi^{(1,1)}(z)|^2=0.
\end{split}\end{equation*}
From this equation, we know Re$(A_i)=-\frac{\mu_j^2}{8}$.
Substituting this to the latter equation of (\ref{aj}) and collecting the imaginary part of its both sizes, we obtain $\mu_j=\mu_k$. Hence by (\ref{aj}) again,
$A_j=A_k$ and $B_j=B_k$.

On the other hand, by [HJY,(4.17)], we know
\begin{equation}\begin{split}\label{92eq3}
&2\Big(-2\ov{z}f^{(1,2)}(z)|z|^2+i\ov{\Phi_0^{(2,0)}(z)}\cdot
{\Phi_0^{(2,1)}(z)}\Big)|z|^2+|\phi^{(3,0)}(z)|^2=0.
\end{split}\end{equation}
Substituting  into this equation, we get
\begin{equation}\begin{split}
&-4\sum\limits_{j=1}^{\kappa_0}
A_j|z_j|^2|z|^4+2i\Big\{\sum\limits_{1\leq j\leq
\kappa_0}\ov{\mu_{jj}z_j^2}\cdot B_jz_j^2+\sum\limits_{1\leq j<k\leq
\kappa_0}\ov{\mu_{jk}z_jz_k}\cdot
(\frac{i}{2}\mu_j\mu_{jk}-\frac{2i\mu_{jk}}{\mu_k}A_k)z_jz_k\\
&+\sum\limits_{1\leq j\leq \kappa_0<\alpha\leq
n-1}\ov{\mu_{j\alpha}z_jz_\alpha}\cdot (-\frac{2i}
{\sqrt{\mu_{j}}}A_j)z_jz_\alpha\Big\}|z|^2+|\Phi_1^{(3,0)}(z)|^2=0.
\end{split}\end{equation}
After a direct simplification, we obtain:
\begin{equation}\begin{split}
|\Phi_1^{(3,0)}(z)|^2=(\sum\limits_{1\leq j\leq
\kappa_0}\mu_{j}|z_j|^2)^2|z|^2.
\end{split}\end{equation}
This means that $ (\Phi_1^{(3,0)}(z))$ is a vector of dimension
$\frac{1}{2}\kappa_0(\kappa_0+1)(n-\kappa_0)+\frac{1}{6}\kappa_0(\kappa_0+1)(\kappa_0+2)$.
Hence the third degeneracy dimension is of
\begin{equation}\begin{split}
&n+(n-1)+\cdots+(n-\kappa_0)+\frac{1}{2}\kappa_0(\kappa_0+1)(n-\kappa_0)+\frac{1}{6}\kappa_0(\kappa_0+1)(\kappa_0+2)\\
=&\frac{\kappa_0}{6}\Big(3(\kappa_0+3)n-(\kappa_0+1)(2\kappa_0+1)\Big).
\end{split}\end{equation}

\end{proof}




\section{Proof of the main theorem}

With (\ref{pqp****}) and Theorem \ref{phi110} at our disposal, we
are now in a position to prove our main theorem.

\begin{proof}[Proof of our main theorem]
We prove the theorem by
absurdity. Suppose that the geometric rank of the map is
$\kappa_0\in[1,n-2]$ and
the map satisfies the degeneracy rank or the degeneracy dimension conditions.
We will prove that the CR Gauss map of the map $F$ must be
non-degenerate. Write
$\mathcal{M}=\{p\in \p\mathbb{B}^n: Rk_F(p)=
\kappa_0\}.$ It will suffice to show the non-degeneracy of the Gauss map for $p_0\in \mathcal M$. This is because, if so, $dim_{\RR}\gamma(\mathcal M)= 2n-1$ while $dim_{\RR}\gamma(\partial\BB^n\setminus \mathcal M)\le dim_{\RR}\partial\BB^n\setminus \mathcal M\le 2n-2$, which would be non-generic. Due to Theorem \ref{thm3} and Theorem \ref{phi110}, we
also suppose at $p_0\in\mathcal M$, $F$ has the form (\ref{fp****}) and has the additional condition
$\phi_{p_0}^{****(1,1)}\not\equiv0$. Without loss of generality, we
suppose that $p_0=0$.

For every $p$ close to $0$, write
\begin{equation}\begin{split}
\frac{\p \phi^{***(2)}_{p,kl}}{\p z_j}
=&\sum\limits_{h=1}^{n-1}\Gamma_{j,kl}^{[h]}(p)z_h+\Gamma_{j,kl}^{[n]}(p)w
+ O(2) ,\
T\phi^{***(2)}_p=\sum\limits_{h=1}^{n-1}\Gamma_{n,kl}^{[h]}(p)z_h+\Gamma_{j,kl}^{[n]}(p)w
+ O(2).
\end{split}\end{equation}
Denote by $\Upsilon(p)$ the following $n (N-n)\times n$ matrix
\begin{equation}\begin{split}
\Upsilon(p)=\big(\Gamma_{j,kl}^{[1]}(p)\ \Gamma_{j,kl}^{[2]}(p)\
\cdots\ \Gamma_{j,kl}^{[n]}(p)\big)_{1\leq j\leq n,(k,l)\in
\mathcal{S}}.
\end{split}\end{equation}
By our normalization properties, we know, for $1\leq j\leq \kappa_0$
and $\kappa_0+1\leq \alpha\leq n-1$, the following
\begin{equation}\begin{split}
\frac{\p}{\p z_{1}}\phi^{(2)}_{1\alpha}=\mu_{1\alpha}z_{\alpha},\
\frac{\p}{\p z_{\alpha}}\phi^{(2)}_{j\alpha}=\mu_{j\alpha}z_{j}.
\end{split}\end{equation}
If $e_{j,kl}\neq 0$, then $\frac{\p}{\p
z_{j}}\phi^{(2)}_{kl}=\sum\limits_{h=1}^{n-1}\lambda_hz_h+\tau w$
for some $\tau\neq 0$. Hence we must have $\Upsilon(0)$ is of rank
$n$. Notice that $ \Upsilon(p)=\Upsilon(0)+o(1) $, hence when $p$ is
sufficiently close to $0$, $\Upsilon(p)$ is also of rank $n$.

On the other hand, we know
$$
\frac{\p \phi^{****}_p}{\p z_j}(q^{****}-1),\ \phi^{****}_p\frac{\p
q^{****}}{\p z_j},\ T\phi^{****}_p(q^{****}-1),\
\phi^{****}_pTq^{****}=O(|(z,w)|^2).
$$
Hence by the implicit function theory, there is a small neighborhood
$U$ of $0$, such that for every point $p\in U$, (\ref{pq}) has only
one solution in some neighborhood $V$. Choose $U_1\subset U$
sufficiently small such that the solutions are contained in $V_1$
which  is also very small and
$(z+\hat{z},w+\hat{w}+2i\ov{\hat{z}}z)\in U\cap V$ for
$(\hat{z},\hat{w})\in U_1$ and $(z,w)\in V_1$. Then for generic
point $(z_0,w_0)\in U_1$, (\ref{pq}) has only one solution  $(z,w)$
with $(\w{z},\w{w})\in U$, which means that the CR Gauss map is
non-degenerate.
\end{proof}

{\bf Acknowledgement} We would like to thank  Xiaojun Huang and  Shanyu Ji for helpful discussions.

\medskip \medskip

\bibliographystyle{amsalpha}

\begin{thebibliography}{A}
\bibitem [A77]{A77} H. Alexander, {\it Proper holomorphic maps in
${\mathbf C}^n$}, Indiana Univ. Math. J. 26 (1977), 137-146.

\bibitem [CJ]{CJ} X. Cheng and S. Ji, {\it Linearity and second fundamental forms for proper holomorphic maps from $\mathbb{B}^{n+1}$ to $\mathbb{B}^{4n-3}$}, J. Geom. Anal. 22 (2012), no. 4, 977-1006.

\bibitem [CJL]{CJL} X. Cheng, S. Ji and W. Liu, {\it CR submanifolds in a sphere and their Gauss maps},
 Sci. China Math. 56 (2013), 1041-1049.

\bibitem[CM]{CM} S. S. Chern and J. K. Moser, {\it Real hypersurfaces in complex manifolds}, Acta Math. 133 (1974), 219-271.


\bibitem [CS2]{CS2} J. Cima and T. J. Suffridge, {\it Boundary behavior of rational
proper maps}, Duke Math. J. 60 (1990), 135-138.



\bibitem [DL] {DL} J. P. D'Angelo and J. Lebl,  {\it  Complexity results
for CR mappings between spheres}, International Journal of
Mathematics, 20 (2009), no. 2, 149-166.

\bibitem [Eb13]{Eb13}
Ebenfelt P. {\it Partial rigidity of degenerate CR embeddings into
spheres}, Advances in Math, 239 (2013),  72-96.

\bibitem[Eb]{Eb} P. Ebenfelt, {\it On the HJY gap conjecture in CR geometry vs. the SOS conjecture for polynomials}, Analysis and geometry in several complex variables, 125-135, Contemp. Math., 681, Amer. Math. Soc., Providence, RI, 2017.



\bibitem[EHZ]{EHZ} P. Ebenfelt, X. Huang and D. Zaitsev, {\it Rigidity of CR-immersions into spheres}, Comm. Anal. Geom.  12  (2004),  no. 3, 631-670.


\bibitem [Fa86]{Fa86} J. Faran, {\it The linearity of proper holomorphic
maps between balls in the low codimension case}, J. Diff. Geom.
24 (1986), 15-17.

\bibitem [FHJZ10]{FHJZ10} J. Faran, X. Huang, S. Ji and Y. Zhang, {\it Rational and polynomial
maps between balls}, Pure and Applied Mathematics Quarterly,
6 (2010), no. 3, 829-842.


\bibitem[Fo92]{Fo92}
F. Forstneric, {\it A survey on proper holomorphic
            mappings}, Proceeding of Year in SCVs at Mittag-Leffler
            Institute, Math. Notes 38(1992), Princeton University
            press, Princeton, N.J.


\bibitem[GH]{GH} P. Griffiths and J. Harris, {\it Algebraic geometry and local differential geometry}, Ann. Sci. \'Ecole Norm. Sup. (4) 12 (1979), no. 3, 355-452.


\bibitem [Ha05] {Ha05}
H. Hamada, {\it Rational proper holomorphic maps from ${\bf{B}}^n$
into ${\bf{B}}^{2n}$},  Math. Ann.  331 (2005), no.3, 693-711.

\bibitem[Hu99]{Hu99}X. Huang, {\it On a linearity of proper holomorphic maps between
balls in the complex spaces of different dimensions}, J. Diff.
Geom, 51 (1999), 13-33.


\bibitem[Hu03]{Hu03}X. Huang, {\it On a semi-rigidity property for holomorphic maps},
Asian J. Math, 7 (2003), 463-492.

\bibitem[HJ01]{HJ01} X. Huang and S, Ji, {\it Mapping $\mathbb{B}^{n}$ into
$\mathbb{B}^{2n-1}$}, Invent. Math, 145 (2001), 219-250.


 \bibitem[HJX06]{HJX06} X. Huang, S. Ji and D. Xu, {\it A new gap
  phenomenon for proper holomorphic mappings from $\BB^n$ to
  $\BB^N$}, Math Research Letter, 3 (2006), no. 4, 509-523.

  \bibitem [HJY09]{HJY09}
X. Huang, S. Ji and W. Yin, {\it Recent Progress on Two Problems
in Several Complex Variables}, Proceedings of
International Congress of Chinese Mathematicians 2007, vol. I,
563-575, International Press, 2009.

\bibitem [HJY14]{HJY14}
X. Huang, S. Ji and W. Yin, {\it On the Third Gap
  for Proper Holomorphic Maps between Balls}, Math. Ann.  358 (2014), 115-142.

\bibitem[Hw]{Hw} J. M. Hwang, {\it Varieties with degenerate Gauss mappings in complex hyperbolic space forms}, Internat. J. Math. 13 (2002), no. 2, 209-216.

\bibitem[IL]{IL}T. A. Ivey and J. M. Landsberg, {Cartan for beginners. Differential geometry via moving frames and exterior differential systems}, Second edition, Graduate Studies in Mathematics, 175. American Mathematical Society, Providence, RI, 2016. xviii + 453 pp. ISBN: 978-1-4704-0986-9.

    \bibitem [Lam01] {L01a}
Lamel B. {\it A reflection principle for real-analytic submanifolds of
complex spaces,} J. Geom. Anal. 4 (2001), no. 11, 625-631.

\bibitem[JX04]{JX04} S. Ji and D. Xu, {\it Maps between $\mathbb{B}^{n}$ and
$\mathbb{B}^{N}$ with geometric rank $k_0\leq n-2$ and minimum $N$},
Asian J. Math, 8 (2004), 233-258.

\bibitem[JY]{JY} S. Ji, and Y. Yuan, {\it Flatness of CR submanifolds in a sphere}, Sci. China Math. 53 (2010), no. 3, 701-718. (Special Issue Dedicated to Professor Yang Lo)




\bibitem[R]{R} Z. Ran, {\em On subvarieties of abelian varieties}, Invent. Math. 62 (1981), no. 3, 459-479.


\bibitem[W1]{W1} S. Webster, {\it Pseudo-Hermitian structures on a real hypersurface}, J. Differential Geom. 13 (1978), no. 1, 25-41.

\bibitem[W2]{W2} S. Webster, {\it On mapping an (n+1)-ball in the complex
space}, Pac. J. Math. 81 (1979), 267-272.

\end{thebibliography}

\noindent Wanke Yin, wankeyin@whu.edu.cn, School of Mathematics and Statistics, Wuhan University,  Wuhan, Hubei 430072,
China.\\

\noindent Yuan Yuan, yyuan05@syr.edu, Department of Mathematics, Syracuse University, Syracuse, NY 13244 USA.\\

\noindent Yuan Zhang, zhangyu@pfw.edu, Department of Mathematics, Purdue University, Fort Wayne, IN 46805, USA.

\end{document}